\documentclass[11pt]{amsart}

\usepackage[T1]{fontenc}
\usepackage{amsmath,amsthm,amssymb}

\newtheorem{theorem}{Theorem}
\newtheorem{lemma}{Lemma}

\newcommand{\dH}{d_H}
\newcommand{\ones}{\mathbf{1}}
\newcommand{\R}{\mathbb{R}}
\DeclareMathOperator{\rank}{rank}

\title[Codes with a single Hamming distance]{An eigenvalue proof of Heged\"{u}s's bound for codes with a single Hamming distance}

\author{Scott Duke Kominers}
\address{Harvard Business School; Department of Economics and Center of Mathematical Sciences and Applications, Harvard University; and a16z crypto}
\email{kominers@fas.harvard.edu}
\thanks{I used LLMs to assist with computations in the preparation of this article, particularly GPT-5.5 Pro and Claude~4.8 Opus (both accessed in part via Poe with the support of Quora, where I am an advisor).  The problem, methods, and eventual written form are my own; and of course any errors remain my responsibility.  This work was conducted while I was visiting the Technological Innovation, Entrepreneurship, and Strategic Management (TIES) Group at the MIT Sloan School of Management; I greatly appreciate their hospitality.}


\begin{document}

\begin{abstract}
We give a short, self-contained linear-algebra proof of a bound of Heged\"{u}s: if all pairwise Hamming distances in a family of subsets of \(\{1,\ldots,n\}\) equal a fixed value \(\lambda\ne(n+1)/2\), then the family has at most \(n\) members. Our proof uses the same Gram matrix as in Heged\"{u}s's argument, but reads its eigenvalues in place of its determinant, and keys off of a single fact about vectors of equal norm and equal pairwise inner product. That fact applies verbatim over an alphabet of size \(q\), where it yields the bound \(n(q-1)\) for \(\lambda\ne\bigl((q-1)n+1\bigr)/q\) --- the corrected form of a conjecture of Heged\"{u}s, recently established by Hu, Huang, and Yu.
\end{abstract}

\maketitle

Heged\"{u}s \cite{Hegedus2026} bounded the size of a family of subsets of \(\{1,\ldots,n\}\) all of whose pairwise Hamming distances are equal; he obtained the bound \(m\le n+1\) from a Delsarte-type inequality \cite{Delsarte1973,BabaiSnevilyWilson}, and then excluded the extremal case by evaluating a determinant based on \cite[Exercise~4.1.3]{BabaiFrankl}. We retain the Gram matrix at the heart of Heged\"{u}s's argument but read its eigenvalues rather than its determinant, which yields both conclusions at once: whereas the determinant detects only whether the Gram matrix is singular, the eigenvalues control the rank in every case --- this finer information makes the Delsarte input unnecessary. The proof reduces to a single fact about vectors of equal norm and equal pairwise inner product (Lemma~\ref{lem:gram}); from there, the set-family bound (Theorem~\ref{thm:main}) and a \(q\)-ary analogue (Theorem~\ref{thm:qary}) both follow directly. The \(q\)-ary bound is the corrected form of a conjecture of Heged\"{u}s \cite{Hegedus2026}; it was recently proven by Hu, Huang, and Yu~\cite{HuHuangYu2025}. Like Heged\"{u}s \cite{Hegedus2026}, Hu, Huang, and Yu \cite{HuHuangYu2025} pair a Delsarte-type size bound with a separate argument for the extremal case; the eigenvalue reading again collapses both steps into a single application of Lemma~\ref{lem:gram}.

Throughout, \(n\) is a positive integer, \([n]=\{1,\ldots,n\}\), and subsets of \([n]\) are identified with \(\{\pm1\}\)-vectors. For \(F,G\subseteq[n]\), let \(\dH(F,G)=|F\triangle G|\) denote the Hamming distance, where \(F\triangle G\) is the symmetric difference.

Our key lemma is the following.

\begin{lemma}\label{lem:gram}
Let \(d\ge1\), and let \(w_1,\ldots,w_m\in\R^d\) have equal norms \(\langle w_i,w_i\rangle=s\) and a common pairwise inner product \(\langle w_i,w_j\rangle=t\) for \(i\ne j\), with \(t<s\). Then \(m\le d+1\); moreover, if \(m=d+1\), then \(t=-s/d\).
\end{lemma}

\begin{proof}
The Gram matrix \(N=\bigl(\langle w_i,w_j\rangle\bigr)_{i,j=1}^m=(s-t)I_m+t\,\ones\ones^{\mathsf T}\) is real symmetric. (Here \(\ones\) is the all-\(1\)s vector and \(I_m\) the \(m\times m\) identity matrix.) Every \(x\perp\ones\) satisfies \(Nx=(s-t)x\), while \(N\ones=\bigl(s+(m-1)t\bigr)\ones\); since \(\ones^{\perp}\) and \(\langle\ones\rangle\) span \(\R^m\), the spectrum of \(N\) consists of \(s-t\) with multiplicity \(m-1\) and \(s+(m-1)t\) with multiplicity \(1\). Because \(N\) is the Gram matrix of the vectors \(w_1,\ldots,w_m\in\R^d\), its rank is the dimension of their span, so \(\rank N\le d\). As \(t<s\), the eigenvalue \(s-t\) is positive and occurs with multiplicity \(m-1\), so \(\rank N\ge m-1\). Hence we see that \(m-1\le d\) --- equivalently, \(m\le d+1\). If \(m=d+1\), then \(\rank N\le d=m-1\) forces \(N\) to be singular; as \(s-t>0\), the remaining eigenvalue must vanish, giving \(s+(m-1)t=s+dt=0\), so that \(t=-s/d\).
\end{proof}

With Lemma~\ref{lem:gram} in hand, Heged\"{u}s's bound follows directly.

\begin{theorem}[Heged\"{u}s \cite{Hegedus2026}, Theorem~1.3]\label{thm:main}
Let \(\mathcal F=\{F_1,\ldots,F_m\}\) be subsets of \([n]\) and suppose there is a positive integer \(\lambda\) such that
\[
  \dH(F_i,F_j)=\lambda \qquad\text{for all } i\ne j.
\]
If \(\lambda\ne\dfrac{n+1}{2}\), then \(m\le n\).
\end{theorem}

\begin{proof}
If \(m\le1\), then there is nothing to prove --- so we assume \(m\ge2\).
Let \(v_i\in\{\pm1\}^n\) denote the signed characteristic vector of \(F_i\), defined by
\[
  (v_i)_j=
  \begin{cases}
    +1 & j\in F_i,\\
    -1 & j\notin F_i.
  \end{cases}
\]
Then \(\langle v_i,v_i\rangle=\lVert v_i\rVert^2=n\). For \(i\ne j\) the vectors \(v_i\) and \(v_j\) disagree in exactly \(\lambda\) coordinates, so \(\langle v_i,v_j\rangle=(n-\lambda)-\lambda=n-2\lambda\). The resulting Gram matrix \(2\lambda I_m+(n-2\lambda)\ones\ones^{\mathsf T}\) is the matrix used in \cite{Hegedus2026}; rather than compute its determinant, we apply Lemma~\ref{lem:gram}. Since \(\lambda\ge1\) gives \(n-2\lambda<n\), the lemma applies with \(d=s=n\) and \(t=n-2\lambda\), yielding \(m\le n+1\); and if \(m=n+1\) then \(n-2\lambda=t=-s/d=-1\), i.e., \(\lambda=\tfrac{n+1}{2}\). Hence \(\lambda\ne\tfrac{n+1}{2}\) forces \(m\le n\).
\end{proof}

\section*{Remarks}

\noindent\textbf{(a) A geometric reading.}
Lemma~\ref{lem:gram} is the classical simplex bound. For the vectors \(v_i\) in the proof of Theorem~\ref{thm:main}, normalizing to \(u_i=v_i/\sqrt{n}\) gives \(m\) unit vectors with common pairwise inner product \(\alpha=(n-2\lambda)/n\) --- that is, \(m\) equidistant points on a sphere in \(\R^n\), hence the vertices of a regular simplex, of which there are at most \(n+1\). The extremal value \(\lambda=\tfrac{n+1}{2}\) is exactly the case \(t=-s/d\) of the lemma, where the points close up into a full-dimensional regular simplex, with pairwise inner product \(-1/n\).

\medskip
\noindent\textbf{(b) The hypothesis on \(\lambda\) is necessary.}
The excluded value \(\lambda=\tfrac{n+1}{2}\) cannot be dropped, since there the conclusion of Theorem~\ref{thm:main} can fail. Take \(n=3\) and \(\lambda=2=\tfrac{n+1}{2}\), and consider the four sets
\[
  \{1,2,3\},\quad \{1\},\quad \{2\},\quad \{3\},
\]
which are at Hamming distance \(2\) from each other, so \(m=4>3=n\). Their signed characteristic vectors
\[
  (+,+,+),\quad (+,-,-),\quad (-,+,-),\quad (-,-,+)
\]
have squared norm \(3\) and every pairwise inner product equal to \(-1\); this is exactly the equality case \(t=-s/d\) of Lemma~\ref{lem:gram}, since here \(-1=-3/3\). (Geometrically these are the four vertices of a regular tetrahedron in \(\R^3\), as in Remark~(a).)

Exceptional families of the form just described exist for infinitely many \(n\). Whenever a Hadamard matrix of order \(n+1\) exists, multiply its rows by \(\pm1\) so that the first column is \(\ones\), and delete that column; the \(n+1\) remaining rows lie in \(\{\pm1\}^n\), and since the original rows are pairwise orthogonal while the deleted entries are all \(+1\), these truncated rows have pairwise inner products all equal to \(-1\). The corresponding \(n+1\) subsets of \([n]\) are then pairwise at distance \(\tfrac{n+1}{2}\), giving \(m=n+1>n\). The Sylvester construction~\cite{Sylvester1867} supplies Hadamard matrices of every order \(2^r\), so this occurs for all \(n=2^r-1\) with \(r\ge2\) --- that is, \(n=3,7,15,\ldots\) --- and hence for infinitely many odd \(n\ge5\). (Conversely, adjoining a column of \(+1\)s to any size-\((n+1)\) family at distance \((n+1)/2\) makes its rows pairwise orthogonal. Thus such extremal families are precisely obtained by normalizing a Hadamard matrix of order \(n+1\) to have one all-\(+1\) column and deleting that column, up to reordering the rows and permuting/sign-changing the remaining columns.) This complements Remark~(c) below: a Hadamard matrix of order \(n\) makes the bound \(m\le n\) sharp, whereas one of order \(n+1\) breaks it at the single forbidden distance.

\medskip
\noindent\textbf{(c) Sharpness.}
For every Hadamard matrix of order \(n\ge2\), the bound \(m\le n\) is sharp: Taking the rows of such a matrix as \(\pm1\)-vectors yields \(n\) subsets. If \(r\) and \(r'\) are two distinct rows, then \(\langle r,r'\rangle=0\). Writing \(a\) for the number of coordinates in which \(r\) and \(r'\) agree and \(b\) for the number in which they differ, we have \(a+b=n\) and \(a-b=0\), hence \(b=n/2\). Thus any two of the corresponding subsets differ in exactly \(n/2\) coordinates; so \(\lambda=n/2\), \(\tfrac{n}{2}\ne\tfrac{n+1}{2}\), and \(m=n\). (Equivalently, after normalizing one row of the Hadamard matrix to be all \(1\)s, this is the standard Hadamard-row calculation of \cite[Thm.~14.9]{Jukna2011}.)

\medskip
\noindent\textbf{(d) Even \(n\).}
Since \(\lambda\) is a positive integer, \((n+1)/2\) is an integer only when \(n\) is odd. Hence for even \(n\) the hypothesis of Theorem~\ref{thm:main} holds automatically and \(m\le n\) is unconditional; by Remark~(c) it is sharp for infinitely many even \(n\) (for instance \(n=2^r\) with \(r\ge1\)). The extremal configurations of Remarks~(a)--(b) can therefore occur only for odd \(n\).

\section*{The \(q\)-ary case}

Lemma~\ref{lem:gram} gives a short proof of a \(q\)-ary analogue of Theorem~\ref{thm:main} as well, matching the bound recently established by Hu, Huang, and Yu~\cite{HuHuangYu2025}. For an integer \(q\ge2\), give \(\{0,1,\ldots,q-1\}^n\) the Hamming distance
\[
  \dH(x,y)=|\{i\in[n]:x_i\ne y_i\}|;
\] the binary case is \(q=2\).

\begin{theorem}[Hu--Huang--Yu \cite{HuHuangYu2025}, Theorem~6]\label{thm:qary}
Consider \(n\ge1\) and \(q\ge2\). Let \(V=\{x_1,\ldots,x_m\}\subseteq\{0,1,\ldots,q-1\}^n\) and suppose there is a positive integer \(\lambda\) such that \(\dH(x_i,x_j)=\lambda\) for all \(i\ne j\). If
\[
  \lambda\ne\frac{(q-1)n+1}{q},
\]
then \(m\le n(q-1)\).
\end{theorem}

\begin{proof}
Again, if \(m\le1\) then there is nothing to prove, so we assume \(m\ge2\). Let \(e_0,\ldots,e_{q-1}\) be the standard basis of \(\R^q\), let \(U=(\ones_q)^{\perp}\) (where \(\ones_q\) is the all-\(1\)s vector in \(\R^q\)), and set \(f_a=\sqrt{q}\,(e_a-\tfrac{1}{q}\ones_q)\in U\) for \(0\le a<q\). A direct check gives \(\langle f_a,f_a\rangle=q-1\) and \(\langle f_a,f_b\rangle=-1\) for \(a\ne b\), so the \(f_a\) are the vertices of a regular simplex centred at the origin; identifying \(U\cong\R^{q-1}\), embed words by \(\phi(x)=(f_{x_1},\ldots,f_{x_n})\in U^n\cong\R^{n(q-1)}\). (For \(q=2\), \(U\cong\R\) and \(f_0,f_1=\pm1\) recover the signed embedding used in the proof of Theorem~\ref{thm:main}.) Then \(\langle\phi(x_i),\phi(x_i)\rangle=n(q-1)\), and for \(i\ne j\) the words agree in \(n-\lambda\) coordinates and disagree in the remaining \(\lambda\), so
\[
  \langle\phi(x_i),\phi(x_j)\rangle=(n-\lambda)(q-1)+\lambda(-1)=n(q-1)-\lambda q .
\]
Since \(\lambda q>0\) gives \(n(q-1)-\lambda q<n(q-1)\), Lemma~\ref{lem:gram} applies with \(d=s=n(q-1)\) and \(t=n(q-1)-\lambda q\), yielding \(m\le n(q-1)+1\); and if \(m=n(q-1)+1\) then \(t=-s/d=-1\) (as in Theorem~\ref{thm:main}, the embedded vectors have \(s=d\), so the extremal condition is simply \(t=-1\)), i.e., \(n(q-1)-\lambda q=-1\) and \(\lambda=\frac{(q-1)n+1}{q}\). Hence \(\lambda\ne\frac{(q-1)n+1}{q}\) forces \(m\le n(q-1)\).
\end{proof}

\medskip
\noindent\textbf{(e) On the conjecture of Heged\"{u}s~\cite{Hegedus2026}.}
Theorem~\ref{thm:qary} is the corrected \(q\)-ary form of Conjecture~1 of Heged\"{u}s~\cite{Hegedus2026}, as proven by Hu, Huang, and Yu \cite{HuHuangYu2025}. As with Heged\"{u}s's original result, the eigenvalue reading makes the \(q\)-ary case a consequence of Lemma~\ref{lem:gram}. As in the binary case, the exceptional value can matter only when it is integral; here this is equivalent to \(n\equiv1\pmod{q}\). The threshold \(\tfrac{(q-1)n+1}{q}\) reduces to Theorem~\ref{thm:main}'s \(\tfrac{n+1}{2}\) when \(q=2\), but for \(q>2\) it differs from the value \(\tfrac{(q-1)(n+1)}{q}\) appearing in~\cite[Conjecture~1]{Hegedus2026}. That the latter is too large is already visible at \(n=1\): the \(q\) words of length \(1\) over the alphabet \(\{0,1,\ldots,q-1\}\) --- corresponding to the regular-simplex equality case of Remark~(a) --- lie pairwise at distance \(1=\tfrac{(q-1)\cdot1+1}{q}\), with \(m=q>q-1=n(q-1)\), so the bound \(m\le n(q-1)\) fails at this distance. Yet for \(q\ge3\), we have \(\tfrac{(q-1)(1+1)}{q}=2-\tfrac{2}{q}>1\), so we see that \(\tfrac{(q-1)(n+1)}{q}\) cannot be the correct threshold.

\providecommand{\bysame}{\leavevmode\hbox to3em{\hrulefill}\thinspace}
\providecommand{\MR}{\relax\ifhmode\unskip\space\fi MR }
\providecommand{\MRhref}[2]{%
  \href{http://www.ams.org/mathscinet-getitem?mr=#1}{#2}
}
\providecommand{\href}[2]{#2}

\end{document}